\documentclass[a4paper,12pt]{amsart}

\usepackage{amssymb, enumitem,mathtools}
\usepackage{color}

\usepackage[breaklinks=true,colorlinks=true,linkcolor=green,citecolor=red,urlcolor=blue]{hyperref}

\allowdisplaybreaks

\newcommand \al{\alpha}
\newcommand \be{\beta}

\newcommand \la{\lambda}

\newcommand \br{\mathbb{R}}
\newcommand \bc{\mathbb{C}}

\newcommand \Span{\operatorname{Span}}
\newcommand \id{\operatorname{id}}
\newcommand \<{\langle}
\renewcommand \>{\rangle}
\newcommand \ip{\< \cdot, \cdot \>}

\newcommand \Tr{\operatorname{Tr}}
\newcommand \rk{\operatorname{rk}}

\newcommand \tM{\overline{M}}

\theoremstyle{plain}

\newtheorem*{theorem*}{Theorem}
\newtheorem*{corollary*}{Corollary}

\newtheorem*{conj*}{Conjecture}

\newtheorem*{lemma*}{Lemma}

\newtheorem*{prop*}{Proposition}

\theoremstyle{definition}

\newtheorem*{definition*}{Definition}

\theoremstyle{remark}

\makeatletter
\@namedef{subjclassname@2020}{%
  \textup{2020} Mathematics Subject Classification}
\makeatother

\begin{document}

\title{On 2-stein submanifolds in space forms}

\author{Yunhee Euh}
\address{Department of Mathematics, Sungkyunkwan University, Suwon, 16419, Korea}
\email{prettyfish@skku.edu}

\author{Jihun Kim}
\address{Department of Mathematics, Sungkyunkwan University, Suwon, 16419, Korea}
\email{jihunkim@skku.edu}

\author{Yuri Nikolayevsky}
\address{Department of Mathematics and Statistics, La Trobe University, Melbourne, Victoria, 3086, Australia}
\email{y.nikolayevsky@latrobe.edu.au}

\author{JeongHyeong Park}
\address{Department of Mathematics, Sungkyunkwan University, Suwon, 16419, Korea}
\email{parkj@skku.edu}

\thanks {The first and fourth authors were supported by Samsung Science and Technology Foundation under Project\\
\indent Number SSTF-BA2001-03. \\
\indent The third author was partially supported by ARC Discovery grant DP210100951.}



\begin{abstract}
We prove that a $2$-stein submanifold in a space form whose normal connection is flat or whose codimension is at most $2$, has constant curvature.
\end{abstract}

\maketitle

\section{Introduction}
\label{s:intro}

Einstein hypersurfaces in spaces of constant curvature are well known: by \cite[Theorem~7.1]{Fia}, any such (connected) hypersurface is either totally umbilical (in particular, totally geodesic), or is of conullity $1$, or is an open domain of the product of two spheres of the same Ricci curvature in the sphere. Departing from this result, one may either consider Einstein hypersurfaces in ``more complicated" spaces, or otherwise, relax the codimension one condition.

Results in the first direction include classification of Einstein hypersurfaces in rank-one symmetric spaces (by six (!) groups of authors, including the third and the fourth author). The answer is that such hypersurfaces only exist in quaternionic projective spaces and in the Cayley projective plane and are geodesic spheres of particular radii. The third and the fourth author further classified Einstein hypersurfaces in irreducible symmetric spaces: in higher rank spaces, such hypersurfaces are either Einstein codimension $1$ solvmanifolds in spaces of noncompact type, or some very special hypersurfaces in the Wu manifold $\mathrm{SU}(3)/\mathrm{SO}(3)$ or in its noncompact dual $\mathrm{SL}(3)/\mathrm{SO}(3)$ \cite{NP}.

On the other hand, for submanifolds of higher codimension in a space form, the Einstein condition alone may be too weak for obtaining a meaningful classification, and one needs to impose extra restrictions of either extrinsic or intrinsic nature (see e.g. \cite{DOV} and many results in \cite{DT}).

In this paper we impose, as such a restriction, the $2$-stein condition. A Riemannian manifold is called \emph{$2$-stein}, if at every point, both the trace of its Jacobi operator and the trace of the squared Jacobi operator are constant on the unit tangent sphere. The $2$-stein condition, being quite algebraically restrictive, is still not fully understood. Algebraically, it roughly determines ``a half" of the curvature tensor in higher dimensions. Spaces of constant curvature, rank-one symmetric spaces and Damek-Ricci spaces (Einstein solvmanifolds of rank one constructed via Clifford modules \cite{BTV}) are examples of $2$-stein manifolds. There are many $2$-stein symmetric spaces (see the classification in \cite{CGW}). Any \emph{harmonic} space is $2$-stein (but not vice versa). Dimension $4$ is special for $2$-stein manifolds (see the first paragraph of Section~\ref{s:pre}); in dimension $5$, any $2$-stein manifold is either a space form, or the Wu manifold $\mathrm{SU}(3)/\mathrm{SO}(3)$ or its noncompact dual $\mathrm{SL}(3)/\mathrm{SO}(3)$ \cite[Proposition~1]{Nik}.

We prove the following.
\begin{theorem*} 
  Let $M^n, \; n \ge 3$, be a $2$-stein, connected submanifold in a space $\tM^{n+p}(c), \; p \ge 1$, of constant curvature $c$. Suppose one of the following two assumptions is satisfied.
    \begin{enumerate}[label=\emph{(\alph*)},ref=\alph*]
    \item \label{it:thflat}
    The normal connection of the submanifold $M^n \subset \tM^{n+p}(c)$ is flat.

    \item \label{it:thcod2}
    The codimension $p$ equals $2$.
  \end{enumerate}
  Then $M^n$ has constant curvature.
\end{theorem*}

Note that in case~\eqref{it:thcod2}, the normal connection is also flat. For the state of knowledge in the theory of submanifolds of constant curvature in spaces of constant curvature the reader is referred to \cite{DT} (especially to Chapter~5) and references therein.

To put case~\eqref{it:thcod2} in context, we note that the image of the standard Veronese embedding of $\bc P^2$, the ``smallest" $2$-stein space of non-constant curvature, given by $z \mapsto z \otimes \overline{z}^t$ (where $z \in \bc^3$ is a unit vector) lies in the sphere $S^7$, hence having codimension $3$.

\section{Preliminaries}
\label{s:pre}

We denote $\ip$ the inner product on $\tM^{n+p}(c)$ and the induced inner product on $M^n$, and $R$ the curvature tensor of $M^n$. For $x \in M^n$ and $X \in T_xM^n$, denote $R_X$ the corresponding Jacobi operator, so that $R_XY=R(Y,X)X$ for $Y \in T_xM^n$. The $2$-stein condition means that there exist two functions $\mu_1, \mu_2$ on $M^n$ such that $\Tr R_X = \mu_1(x) \|X\|^2, \; \Tr (R_X^2) = \mu_2(x) \|X\|^4$, for all $x \in M^n$ and all $X \in T_xM^n$. Of course, $\mu_1$ is a constant when $n > 2$; moreover, there is a similar Schur-type result for $\mu_2$: on a (connected) $2$-stein manifold $M^n$ of dimension $n > 4$, the function $\mu_2$ is a constant \cite[6.61]{Bes}, \cite[Theorem~2.4]{GSV}; this is not true in dimension $4$: see \cite[Section~2]{GSV} and the construction in \cite{CPS}.

Let $A^\sigma$ be the shape operators relative to an orthonormal basis $\xi^\sigma, \; \sigma=1, \dots, p$, in the normal space at a point $x \in M^n$. From Gauss equation we have
\begin{equation}\label{eq:RX}
  R_X = c (\|X\|^2 \id - X \otimes X^\flat) + \sum_{\sigma=1}^{p} (\<A^\sigma X, X\> A^\sigma - (A^\sigma X) \otimes (A^\sigma X)^\flat).
\end{equation}
Note that adding to an algebraic curvature tensor a constant curvature tensor (which is equivalent to adding to the Jacobi operator a ``constant curvature term" $a (\|X\|^2 \id - X \otimes X^\flat), \; a \in \br$) does not affect the property of being $2$-stein, but only changes the constants. Therefore the fact that $M^n$ is $2$-stein is equivalent to the fact that at every point $x \in M^n$, for some constants $c_1', c_2' \in \br$ we have
\begin{equation}\label{eq:2stdash}
  \Tr R_X' = c_1' \|X\|^2, \qquad \Tr ({R_X'}^2) = c_2' \|X\|^4,  \qquad \text{for all } X \in T_xM^n,
\end{equation}
where $R_X'$, the \emph{extrinsic} Jacobi operator, is given by
\begin{equation}\label{eq:RX'}
    R_X' = \sum_{\sigma=1}^{p} (\<A^\sigma X, X\> A^\sigma - (A^\sigma X) \otimes (A^\sigma X)^\flat).
\end{equation}
Substituting the expression for $R_X'$ from \eqref{eq:RX'} into \eqref{eq:2stdash} we obtain
\begin{gather}\label{eq:1stA}
  \sum_{\sigma=1}^{p} (H^\sigma A^\sigma - (A^\sigma)^2) = c_1' \id, \\
  \begin{multlined}[t]
  \sum_{\sigma,\tau=1}^{p} (\<A^\sigma X, X\> \<A^\tau X, X\> T^{\sigma \tau} + \<A^\sigma X, A^\tau X\>^2 \\
  - \<A^\sigma X, X\> \<A^\tau A^\sigma A^\tau X, X\> - \<A^\tau X, X\> \<A^\sigma A^\tau A^\sigma X, X\>) = c_2' \|X\|^4,  \label{eq:2stA}
  \end{multlined}
\end{gather}
for all $X \in T_xM^n$, where $H^\sigma = \Tr A^\sigma$ and $T^{\sigma \tau}=\Tr(A^\sigma A^\tau)$, for $\sigma, \tau=1, \dots, p$.

Note that if all the shape operators $A^\sigma, \sigma=1, \dots, p$, have a common (nontrivial) kernel, then from \eqref{eq:2stA} we have $c_2'=0$, and hence $R_X'=0$ by \eqref{eq:2stdash}.

\section{Proof of the Theorem}
\label{s:pf}

Suppose the submanifold $M^n \subset \tM^{n+p}(c)$ has flat normal connection. Then all the shape operators at any point $x \in M^n$ are simultaneously diagonalisable relative to some orthonormal basis $e_i, \; i=1, \dots, n$, for $T_xM^n$. Denote $\la_i^\sigma=\<A^\sigma e_i, e_i\>$. Then equations~\eqref{eq:1stA} and \eqref{eq:2stA} give
\begin{gather}\label{eq:flEin}
  \sum_{\sigma=1}^{p} (H^\sigma \la_i^\sigma - (\la_i^\sigma)^2) = c_1', \\
  \sum_{\sigma,\tau=1}^{p} (T^{\sigma \tau} \la_i^\sigma \la_j^\tau + \la_i^\sigma \la_j^\sigma \la_i^\tau \la_j^\tau - \la_i^\sigma \la_j^\sigma (\la_j^\tau)^2 - \la_j^\sigma \la_i^\sigma (\la_i^\tau)^2) = c_2', \label{eq:fl2st}
\end{gather}
for all $i,j=1, \dots, n$. In particular, for $j=i$ we obtain
\begin{equation}\label{eq:fl2stii}
    \sum_{\sigma,\tau=1}^{p} (T^{\sigma \tau} \la_i^\sigma \la_i^\tau - (\la_i^\sigma \la_i^\tau)^2) = c_2'.
\end{equation}
Summing up equations~\eqref{eq:fl2st} by $j$ and using \eqref{eq:flEin} we obtain
\begin{equation*}
    \sum_{\sigma,\tau=1}^{p} (T^{\sigma \tau} - H^\sigma H^\tau) \la_i^\sigma \la_i^\tau + 2 c_1' \sum_{\sigma=1}^{p} H^\sigma \la_i^\sigma = n c_2'.
\end{equation*}
Subtracting equation~\eqref{eq:fl2stii} (and using \eqref{eq:flEin} again) we obtain $(n-1)c_2'=(c_1')^2$, and so from \eqref{eq:2stdash}, by Cauchy-Schwartz inequality, we get $R_X'= c_1' (\|X\|^2 \id - X \otimes X^\flat)$, and hence $R_X = (c+c_1^\prime) (\|X\|^2 \id - X \otimes X^\flat)$. This proves the Theorem in the assumption~\eqref{it:thflat}.

\smallskip

Now suppose $p=2$, as in the assumption~\eqref{it:thcod2} of the Theorem. Denote $Q^\sigma = A^\sigma - \frac12 H^\sigma \id, \; \sigma=1,2$. Then the Einstein condition~\eqref{eq:1stA} gives $(Q^1)^2 + (Q^2)^2 = c_3 \id$, for $c_3=\frac14((H^1)^2 + (H^2)^2)-c_1'$. The following fact of linear algebra is probably known (see e.g. \cite[Section~3]{MdF}, in a different notation).

{
\begin{lemma*}
If two symmetric operators $Q^1, Q^2$ on $\br^n$ satisfy the equation $(Q^1)^2 + (Q^2)^2 = c_3 \id$ for some $c_3 \in \br$, then relative to an orthonormal basis for $\br^n$, the matrices of $Q^1$ and $Q^2$ are simultaneously block-diagonal, with all the diagonal blocks being of sizes either $1 \times 1$ or $2 \times 2$, where the $2 \times 2$ pairs of blocks $Q^1_s, \, Q^2_s, \; s=1, \dots, m$, of $Q^1$ and $Q^2$ respectively, have the form
\begin{equation}\label{eq:22blocks}
  Q^1_s=\left(
          \begin{array}{cc}
            \al_s & 0 \\
            0 & -\al_s \\
          \end{array}
        \right), \quad
  Q^2_s=\left(
          \begin{array}{cc}
            \beta_s & \gamma_s \\
            \gamma_s & -\beta_s \\
          \end{array}
        \right),
\end{equation}
with $\al_s^2+\beta_s^2+\gamma_s^2=c_3$ and $\al_s \gamma_s \ne 0$.
\end{lemma*}
\begin{proof}
  Diagonalise $Q^1$. As $Q^2$ and $(Q^1)^2$ commute, $Q^1$ and $Q^2$ simultaneously split into block-diagonal forms, with the corresponding pairs of blocks of the form $\left(\begin{smallmatrix} \al I_p & 0 \\ 0 & -\al I_q \end{smallmatrix}\right)$ and $\left(\begin{smallmatrix} D_1 & T \\ T^t & D_2 \end{smallmatrix}\right)$ respectively, where $T$ is a $p \times q$ matrix, and where different blocks have different $\al^2$. If $\al=0$ or $pq=0$, we can diagonalise the second block. Otherwise, we can separately diagonalise $D_1$ and $D_2$. We have $D_1T+TD_2=0$, which implies that we can further simultaneously split our pairs of blocks into smaller pairs of blocks, each of the form $\left(\begin{smallmatrix} \al I_a & 0 \\ 0 & -\al I_b \end{smallmatrix}\right)$ and $\left(\begin{smallmatrix} \beta I_a & N \\ N^t & -\beta I_b \end{smallmatrix}\right)$ for $Q^1$ and $Q^2$ respectively (where again we can assume $ab \ne 0$). Then $NN^t= c_4 I_a, \; N^tN= c_4 I_b$, where $c_4= c_3-\al^2-\beta^2$. If $c_4=0$, both blocks are diagonal. If not, we must have $a=b$ (if say $a>b$, then $\rk NN^t < a$), and we can diagonalise $N$ by the polar decomposition, without changing any other entries of the blocks. Rearranging the rows and the columns we get $a$ pairs of $2 \times 2$ blocks of the required form~\eqref{eq:22blocks}. If $\al_s \gamma_s = 0$, the blocks $Q^1_s$ and $Q^2_s$ can be simultaneously diagonalised.
\end{proof}
}

We show that, in fact, $m=0$ in the notation of the Lemma (so that there are no $2\times 2$ blocks); then the claim follows from assertion~\eqref{it:thflat}.

Assume there exists at least one pair of $2 \times 2$ blocks, and that an orthonormal basis in the Lemma is chosen in such a way that the corresponding blocks are in the top-left corners (the corresponding invariant subspace of both $A^1$ and $A^2$ is $\Span(e_1,e_2)$). We can further assume that $H^2=0$ by rotating the orthonormal basis $\xi^1, \xi^2$ in the normal space (before applying the Lemma).

We first assume that $H^1 \ne 0$. Taking $X=(x_1,x_2, 0, \dots, 0)^t$ in \eqref{eq:2stA} and equating the coefficients of $x_1^3x_2$ and $x_1x_2^3$ on the left-hand side to zero we obtain $T^{12}=2\alpha\beta$ and $\beta \gamma(2\be^2 + 2\gamma^2 - T^{22})=0$. Equating the coefficients of $x_1^4$ and $x_2^4$ to one another we find $T^{11}=2\alpha^2 + \frac{1}{2}(H^1)^2$, and then equating them to a half of the coefficient of $x_1^2x_2^2$ we obtain
$(\gamma^2-\be^2)(2\be^2 + 2\gamma^2 - T^{22})=0$. It follows that $T^{22}=2\be^2 + 2\gamma^2$. As $T^{22}=\Tr((A^2)^2)$, all the entries of $A^2$
outside the top-left $2\times2$ block must be zero. The same is true for $A^1$, as $T^{11}=2\alpha^2 + \frac{1}{2}(H^1)^2$. But $n \ge 3$, and so $A^1$ and $A^2$ have a common kernel, hence $R'_X=0$ (see the end of Section~\ref{s:pre}).

Next suppose $H^1=0$ (below, we give a direct algebraic proof; alternatively, one can apply the result of \cite{Mat} and then \cite[Lemma~2.1]{CGW}). We have $A^\sigma=Q^\sigma$ for $\sigma=1,2$, and by rotating the basis $\xi^1, \xi^2$ in the normal space we can assume that $T^{12}=0$.

Suppose we have at least two pairs of $2\times2$ blocks of the form \eqref{eq:22blocks}. Up to relabelling we can take $s=1,2$, and assume that the blocks with $s=1$ and $s=2$ correspond to the invariant subspaces $\Span(e_1,e_2)$ and $\Span(e_3,e_4)$ respectively. We also have $\gamma_1 \gamma_2 \ne 0$ by the Lemma. Taking $X=(x_1,x_2, x_3,x_4, 0, \dots, 0)^t$ in \eqref{eq:2stA} and equating the coefficient of $x_1x_2x_3x_4$ on the left-hand side to zero we obtain $T^{22}+
\al_1^2 - \be_1^2 - \gamma_1^2 + \al_2^2 - \be_2^2 - \gamma_2^2=0$. But $T^{22}= \Tr((A^2)^2) \ge 2(\be_1^2+\be_2^2+\gamma_1^2+\gamma_2^2)$, a contradiction.

The last remaining case is when we have exactly one pair of $2\times2$ blocks of the form \eqref{eq:22blocks}. Then, as $n \ge 3$, there is at least one $1\times1$ block, and so up to relabelling we can assume that the top-left $3\times3$ corners of the matrices of $A^1$ and $A^2$ have the form
\begin{equation}\label{eq:33blocks}
  \left(
          \begin{array}{ccc}
            \al & 0 & 0\\
            0 & -\al & 0 \\
            0 & 0 & \la
          \end{array}
        \right), \quad
  \left(
          \begin{array}{ccc}
            \beta & \gamma & 0 \\
            \gamma & -\beta & 0 \\
            0 & 0 & \mu
          \end{array}
        \right),
\end{equation}
respectively. Note that $\al \gamma \ne 0$ by the Lemma and that $\al^2+\be^2+\gamma^2=\la^2+\mu^2$ by \eqref{eq:1stA}. Take $X=(x_1,x_2, x_3,0, \dots, 0)^t$ in \eqref{eq:2stA}. Equating the coefficient of $x_1^3x_2$ on the left-hand side to zero we get $4\be(2\alpha^2 + 2\beta^2 + 2\gamma^2 - T^{22})=0$. If $\be \ne 0$, we get $T^{22}=2\alpha^2 + 2\beta^2 + 2\gamma^2$, and then $T^{11}=2\alpha^2 + 2\beta^2 + 2\gamma^2$ by equating the coefficient of $x_1^4$ to a half of the coefficient of $x_1^2x_2^2$. But then equating the coefficients of $x_1^4$ and $x_3^4$ and using the fact that $\al^2+\be^2+\gamma^2= \la^2 + \mu^2$ we get $\al \gamma = 0$, a contradiction. Suppose $\be = 0$. Equating the coefficient of $x_1x_2x_3^2$ to zero we obtain
$\mu(T^{22} + \alpha^2 - \gamma^2 - \lambda^2 - \mu^2)=0$. But $\la^2+\mu^2=\al^2+\gamma^2$ and $T^{22} \ge 2\gamma^2 + \mu^2$, so $\mu = 0$. Then $\la^2=\al^2+\gamma^2$. Comparing the coefficients of $x_1^4$ and $x_3^4$ we get $T^{11}=4\al^2+2\gamma^2$, while comparing the coefficients of $x_1^2x_3^2$ and $x_2^2x_3^2$ we obtain $T^{11}=2\al^2$, a contradiction.

\end{document}